\documentclass[11pt,reqno]{amsart} 
\usepackage[utf8]{inputenc}
\usepackage{amsmath, amsthm, amsfonts,color,url,enumitem}
\usepackage{cancel}
\usepackage{wrapfig,hyperref}
\usepackage{physics}
\usepackage{mathtools}
\usepackage{fullpage}
\usepackage{subfigure}
\usepackage{color}
\usepackage{thmtools}
\usepackage{thm-restate}
\usepackage{hyperref}
\usepackage{cleveref}
\usepackage[normalem]{ulem}
\usepackage{tikz}
\usepackage{dsfont}
\usepackage{physics}
\usepackage{floatrow}
\usepackage{url}
\usepackage{graphicx}
 \usepackage{diagbox}
\newfloatcommand{capbtabbox}{table}[][\FBwidth]

\numberwithin{equation}{section}

\theoremstyle{definition} 
\newtheorem{theorem}{Theorem}
\newtheorem{definition}{Definition}
\newtheorem{lemma}{Lemma}
\newtheorem{example}{Example}

\newtheorem{corollary}{Corollary}
\newtheorem{remark}{Remark}

\newtheorem{proposition}{Proposition}

\newcommand{\unif}[1]{{\rm Uni}_{#1}}
\newcommand{\PP}{\mathrm{P}}
\newcommand{\TV}{\mathrm{TV}}
\newcommand{\PF}[1]{\mathrm{PF}_{#1}}
\newcommand{\Sym}{\mathfrak{S}}

\newtheorem{innercustomthm}{Theorem}
\newenvironment{customthm}[1]
  {\renewcommand\theinnercustomthm{#1}\innercustomthm}
  {\endinnercustomthm}

\title{Probabilistic Parking Functions}

\author[Durmi\'{c}]{Irfan Durmi\'{c}}
\address[Durmi\'{c}, Han]{Mathematics and Statistics, Williams College, Williamstown, MA 01267}
\email{\textcolor{blue}{\href{mailto:id5@williams.edu}{id5@williams.edu}},\textcolor{blue}{\href{mailto:alh5@williams.edu}{alh5@williams.edu}}}

\author[Han]{Alex Han}

\author[Harris]{Pamela E. Harris}
\address[Harris]{Department of Mathematical Sciences, University of Wisconsin-Milwaukee, Milwaukee, WI 53211}
\email{\textcolor{blue}{\href{mailto:peharris@uwm.edu}{peharris@uwm.edu}}}
\thanks{P.~E.~Harris was supported through a Karen Uhlenbeck EDGE Fellowship.}

\author[Ribeiro]{Rodrigo Ribeiro}
\address[Ribeiro, Yin]{Department of Mathematics, University of Denver, Denver, CO 80208}
\email{\textcolor{blue}{\href{mailto:rodrigo.ribeiro@du.edu}{rodrigo.ribeiro@du.edu}},\textcolor{blue}{\href{mailto:mei.yin@du.edu}{mei.yin@du.edu}}}

\author[Yin]{Mei Yin}
\thanks{M.~Yin was supported by the University of Denver's Faculty Research Fund 84688-145601 and Professional Research Opportunities for Faculty Fund 80369-145601.}

\begin{document}

\begin{abstract}

We consider the notion of classical parking functions by introducing randomness and a new parking protocol, as inspired by the work presented in the paper ``Parking Functions: Choose your own adventure,'' (arXiv:2001.04817)  by Carlson, Christensen, Harris, Jones, and Rodr\'iguez. 
Among our results, we prove that the probability of obtaining a parking function, from a length $n$ preference vector, is independent of the probabilistic parameter $p$.
We also explore the properties of a preference vector given that it is a parking function and discuss the effect of the probabilistic parameter $p$. Of special interest is when $p=1/2$, where we demonstrate a sharp transition in some parking statistics.
We also present several interesting combinatorial consequences of the parking protocol. 
In particular, we 
provide a combinatorial interpretation for the array described in OEIS A220884 as the expected number of preference sequences with a particular property related to occupied parking spots,  which solves an open problem of Novelli and Thibon posed in 2020 (arXiv:1209.5959). Lastly, we connect our results to other weighted phenomena in combinatorics and provide further directions for research.  
\end{abstract}
\maketitle

\section{Introduction} \label{sec:introduction}
The classical parking process occurs on a one-way street with $n\in\mathbb{N}=\{1,2,3,\ldots\}$ parking spots numbered in sequence from one to $n$. In this setting, $n$ cars enter the street, from left to right, each having a parking preference, which we compile in a vector $\alpha=(a_1,a_2,\ldots,a_n)\in[n]^n$, where $[n]=\{1,2,\ldots,n\}$. Each car enters the street and proceeds to their preferred parking spot.
If they find it unoccupied they park there, else they continue down the street parking in the next available spot. 
If under this parking protocol all of the cars are able to park, then we say that the preference vector $\alpha$ is a parking function of length $n$.
For example, any preference vector that is a permutation is a parking function as every car parks in its preferred parking spot. In addition the all ones vector $\alpha=(1,1,\ldots,1)\in[n]^n$ is also a parking function in which, for each $1\leq i\leq n$, car $i$ parks in spot $i$. However, not every parking preference is a parking function, take for example the vector $(n,n,\ldots,n)$ in which case the first car parks in its preferred spot $n$, while the remaining $n-1$ cars, having the same preference, find their spot occupied and then exit the street. 

First explored by Konheim and Weiss in the 1960s \cite{Konheim}, classical parking functions have received considerable attention by mathematical researchers. If $\PF{n}$ denotes the set of all parking functions of length $n$, then it is well-known that $|\PF{n}|=(n+1)^{n-1}$. Many generalizations of parking functions exist and include considering the case where there are more parking spots than cars, considering cases where cars prefer an interval of parking spots rather than a single parking spot, and others include changing the parking procedure and allowing cars to do something different whenever they find their preferred spot occupied. 
For an expository article with recent finding and many open problems related to parking functions we recommend  \cite{carlson2020parking}.

We remark that although parking functions, and their generalizations, are well-studied there has been less work done on considering a probabilistic parking protocol. 
With this point of view in mind, we provide some initial results involving a probabilistic parking protocol. Our initial result considers the following probabilistic scenario:  Fix $p\in[0,1]$ and consider a coin which flips to heads with probability $p$ and tails with probability $1-p$. Then, our parking protocol proceeds as follows: If a car arrives at its preferred spot and finds it unoccupied it parks there. 
If instead the spot is occupied, then the driver tosses the biased coin. If the coin lands on heads, with probability $p$, the driver continues moving forward in the street. However, if the coin lands on tails, with probability $1-p$, the car moves backwards and tries to find an unoccupied parking spot.
Given this new probabilistic parking function protocol we determine the likelihood that a preference vector $\alpha\in[n]^n$ is a parking function. We also explore the properties of $\alpha$ given that it is a parking function and demonstrate the effect of the parameter $p$ (and lack thereof) on the parking protocol. 
Before proceeding further, we note a parking symmetry imposed by this probabilistic parking protocol: Having $n$ cars enter the street from left to right with preference vector $\alpha=(a_1, \dots, a_n)$ and park under protocol with parameter $p$ depicts the same scenario as having $n$ cars enter the street from right to left with preference vector $\alpha'=(n+1-a_1, \dots, n+1-a_n)$ and park under protocol with parameter $1-p$.

Due to the probabilistic nature of our model, we offer a word of caution in interpreting our results. 
In the case of classical parking functions, a preference vector $\alpha \in [n]^n$ is either deterministically a parking function or not, and we can talk about the set of parking functions $\PF{n} \subseteq [n]^n$ and its cardinality $|\PF{n}|$. 
Contrarily, by incorporating a probabilistic parking protocol, all preference vectors $\alpha \in [n]^n$ have a positive probability of being a parking function. 
In stating our results, $\alpha \in \PF{n}$ depicts the situation that $n$ cars with preference vector $\alpha$ park, and we are interested in the probability that this happens and its further implications. 

Our main results are the following.\footnote{While writing a first draft of this work, we became aware of parallel, independent work by Tian and Trevino \cite{tian2020generalizing}. Aside from both introducing the probabilistic parking protocol, the results of the two papers are essentially disjoint, except that Theorem \ref{theorem:pollakcoin} appeared as Theorem 4 in \cite{tian2020generalizing} with a different proof.}

\begin{theorem}\label{theorem:pollakcoin}
Consider the preference vector $\alpha \in [n]^n$, chosen uniformly at random. Then
\begin{eqnarray}\label{eqn:pollakcoin}
\PP(\alpha \in \PF{n} | \alpha \in [n]^n) = \frac{(n+1)^{n-1}}{n^n}.
\end{eqnarray}
\end{theorem}

We emphasize that this result is rather surprising. 
In Example \ref{ex:3cars} we show that the probability of a particular preference vector $\alpha \in [n]^n$ being a parking function does depend on $p$. 
What is magical is that the probabilities of being a parking function for all preference vectors add up in a way so that dependence on $p$ is canceled and there is invariance to the forward probability $p$ for the randomly selected vector. 
In Section~\ref{subsec:coin-m-n}, we show that this invariance generalizes further for $m \in [n]$ cars on a one-way street with $n$ spots. That is we establish the following. 

\begin{theorem}\label{thm:m-n}
Consider the preference vector $\alpha \in [n]^m$, chosen uniformly at random. Then
\begin{eqnarray}\label{eqn:m-n}
\PP(\alpha \in \PF{n} | \alpha \in [n]^m) = \frac{(n+1-m)(n+1)^{m-1}}{n^m}.
\end{eqnarray}
\end{theorem}

Sections \ref{sec:impactppin} and \ref{sec:rate} are devoted to the natural follow-up questions: \emph{Does introducing randomness to the parking protocol achieve anything new? What kind of parking statistics does depend on the probabilistic parameter $p$?} In \ref{sec:impactppin} we present one such parking statistic, namely the parking preference of the last car $a_n$, and explore its statistical properties. In Theorem \ref{component} we give the distribution of $a_n$ and in Theorem \ref{mean} we calculate the asymptotics of its expected value, and the formulas show explicit $p$ dependence.
\begin{theorem}\label{component}
Consider the preference vector $\alpha=(a_1,a_2,\ldots,a_n) \in [n]^n$, chosen uniformly at random. Then given that $\alpha \in \PF{n}$,
\begin{align}
&\PP(a_n=j|\alpha \in \PF{n})=\frac{2}{n+1}-\frac{1}{(n+1)^{n-1}} \Big[p\sum_{s=n-j+1}^{n-1} \binom{n-1}{s} (n-s)^{n-s-2} (s+1)^{s-1}\notag \\
&\hspace{5cm}+(1-p) \sum_{s=0}^{n-j-1} \binom{n-1}{s} (n-s)^{n-s-2} (s+1)^{s-1}\Big].
\end{align}
\end{theorem}
\begin{theorem}\label{mean}
Take $n$ large. For preference vector $\alpha=(a_1,\ldots,a_n) \in [n]^n$ chosen uniformly at random, we have
\begin{equation}
\mathbb{E}(a_n | \alpha \in \PF{n})=\frac{n+1}{2}-(2p-1)\Big[\frac{\sqrt{2\pi}}{4}n^{1/2}-\frac{7}{6}\Big]+o(1).
\end{equation}
\end{theorem}
Based on these results, in Section \ref{sec:rate}, we derive the rate of convergence of the distribution of $a_n$ to uniform distribution on $[n]$ in total variation ($\TV$) distance. In the classical parking model (corresponding to $p=1$), this rate of convergence was shown to be upper bounded by $O(1/\sqrt{n})$ in \cite{bellin2021}. We significantly extend this result and find the exact rate of convergence for any $p \in [0,1]$. 
Specifically, we show that for generic $p \neq 1/2$ the convergence rate is $\Theta(1/\sqrt{n})$, while for $p=1/2$ the convergence rate is $\Theta(1/n)$. 
Note that a consequence of Theorem \ref{component} is that the conditional distribution of $a_n$ for a generic $p$ can be written as a convex combination of the cases $p=0$ and $p=1$. 
This convexity argument will be central to the proof of Theorem \ref{thm:other}.
\begin{theorem}\label{thm:conv12} Let $Q_{n,p}(\cdot)$ be $\PP(a_n = \cdot \;|\; \alpha \in \PF{n})$ under the probabilistic parking protocol with parameter $p$. For $p=1/2$, the following bound holds
    \begin{equation}
    \| Q_{n,1/2}- \unif{n} \|_{\TV} = \Theta\left(\frac{1}{n}\right).
    \end{equation}
\end{theorem}
\begin{theorem}\label{thm:other} Let $Q_{n,p}(\cdot)$ be $\PP(a_n = \cdot \;|\; \alpha \in \PF{n})$ under the probabilistic parking protocol with parameter $p$. For $p\neq 1/2$, we have that 
    \begin{equation}
    \| Q_{n,p}- \unif{n} \|_{\TV} = \Theta\left(\frac{1}{\sqrt{n}}\right).
    \end{equation}
\end{theorem}
The much faster convergence to uniform when $p=1/2$ is reflected in our simulations. Figure \ref{distribution} shows a histogram of the values of $a_n$ based on $100,000$ random samples for $n=100$, with varying probabilistic parameter $p$.
\begin{figure}
\centering
\subfigure{\includegraphics[width=3in]{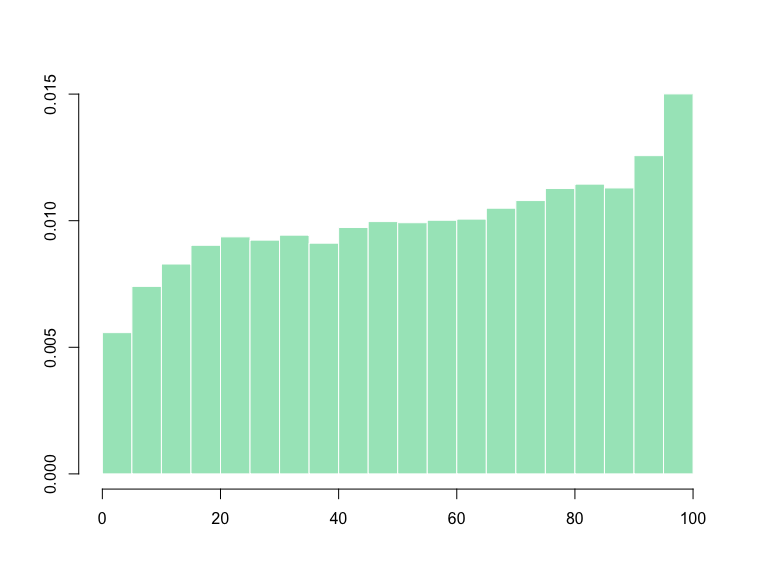}}
\hfill
\subfigure{\includegraphics[width=3in]{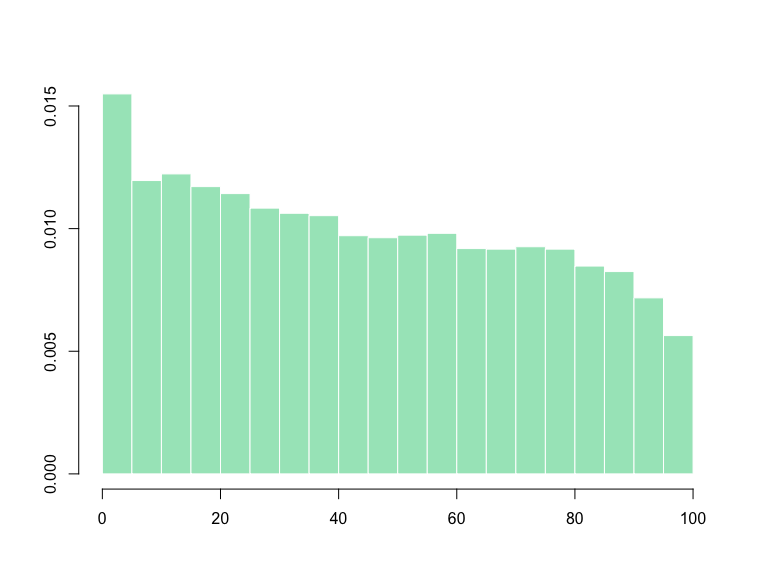}}
\vskip\baselineskip
\subfigure{\includegraphics[width=3in]{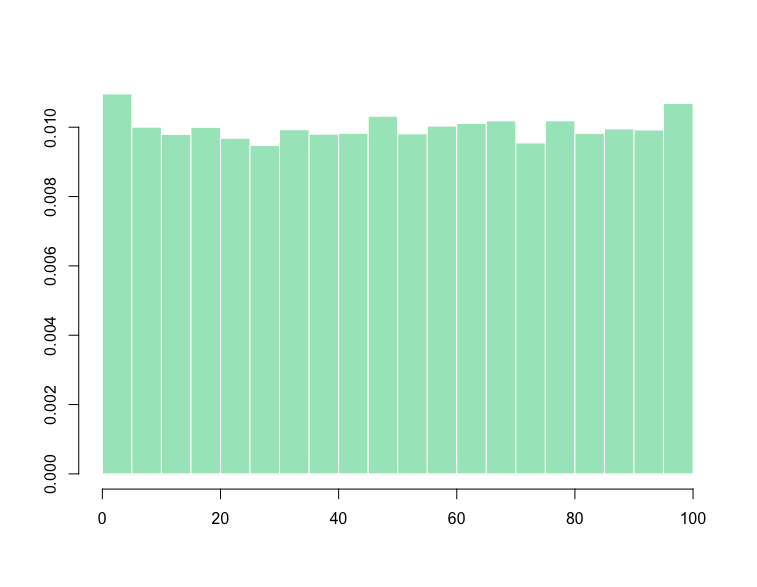}}
\vskip\baselineskip
\subfigure{\includegraphics[width=3in]{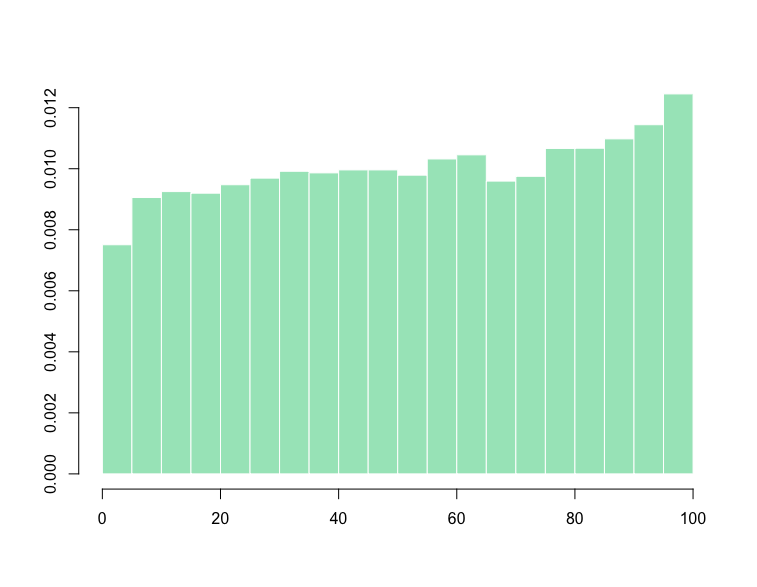}}
\hfill
\subfigure{\includegraphics[width=3in]{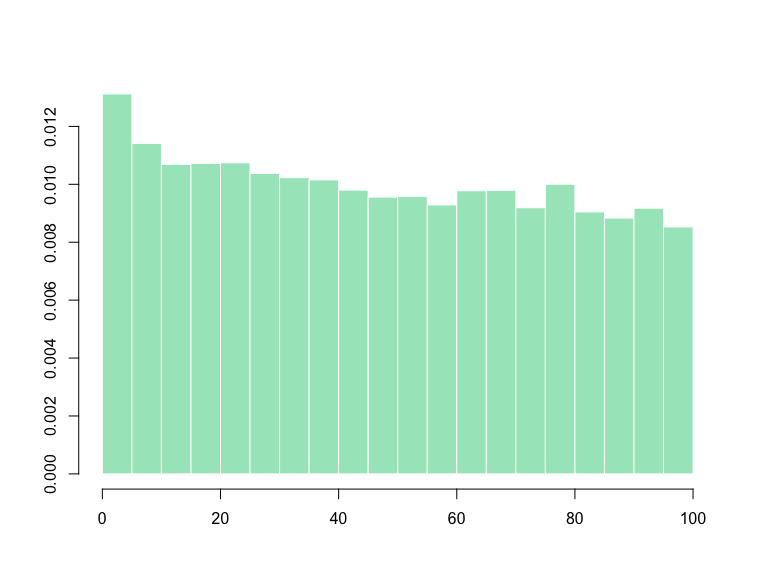}}
\caption{
The conditional distribution of $a_n$ (parking preference of the last car) in $100,000$ samples of preference vectors of size $100$ chosen uniformly at random. The upper left plot is for $p=0$ and the upper right plot is for $p=1$. The middle plot is for $p=0.5$. The lower left plot is for $p=0.25$ and the lower right plot is for $p=0.75$. Note the preference symmetry between $p$ and $1-p$ as observed earlier in the introduction.}
\label{distribution}
\end{figure}
We give some intuitive explanations for this interesting phenomenon. As $p$ increases from $0$ to $1/2$, the distribution of $a_n$ places less and less mass on cars with large parking preferences, as cars are gradually losing their backwards moving bias when their desired spot is taken. 
Similarly, as $p$ decreases from $1$ to $1/2$, the distribution of $a_n$ places less and less mass on cars with small parking preferences, as cars are gradually losing their forward moving bias when their desired spot is taken. 
For $p=1/2$, the movement bias of a car when the desired spot is taken is gone because going left and right have the same probability, so the distribution of $a_n$ is much closer to uniform distribution.

In Section~\ref{section:relatedresults}, we then present several combinatorial consequences of the probabilistic parking protocol. In particular, we relate this to the notion of lucky cars in Section~\ref{subsections:lucky}, inspired by Gessel and Seo \cite[Page 1]{Gessel}, and prove the following.

\begin{theorem}\label{thm:sumsofprods}
Suppose $n$ cars park on a circle with spots $[n+1]$. The expected number of preference vectors in $[n+1]^n$ containing $0 \le k < n$ unlucky cars is
\begin{eqnarray}\label{eqn:sumsofprods}
U_n(k) = n! \sum\limits_{C_k \vdash \{2,\dots,n\}} \prod\limits_{c \in C_k} \frac{i-1}{(n+1)-(i-1)},
\end{eqnarray}
\noindent where $C_k$ is a size-$k$ subset of the cars $\{2,\dots,n\}$, and each car $i \in C_k$ is unable to park at its preferred spot.
\end{theorem}

Then, in Section~\ref{subsection:oeis}, we further explore Theorem~\ref{thm:sumsofprods} and present an interesting connection to an open problem proposed by Novelli and Thibon \cite{novelli2020duplicial} involving the integer sequence OEIS \href{https://oeis.org/A220884}{A220884}, summarized by the following.

\begin{theorem}\label{thm:oeistriangle}
Let $Q_n(q)$ be the generating polynomial $\prod_{k=2}^n [(n+1-k)q + k]$. Then
\begin{equation}
Q_n(q) = \sum_{k=0}^{n-1} U_n(k) q^k.
\end{equation}
\end{theorem}

Finally, in Section~\ref{subsection:pascal}, we present an interesting connection between Pascal's triangle and a recurrence relation which counts the expected number of preference sequences with a particular property related to occupied parking spots. This recurrence relation has the form:
\begin{equation}
E_n(i,k)=((n+1)-(i-1))\cdot E_n(i-1,k)+(i-1)\cdot E_n(i-1,k-1),
\end{equation}
where $E_n(i,k)$ denotes the expected number of preference vectors in $[n+1]^i$
containing $k$ unlucky cars, for some $0 \le k < i \le n$. Note that Gessel and Seo~\cite{Gessel} define a car as lucky if it is able to park at its preferred spot.

This paper is organized as follows. In Section~\ref{sec:prob-intro}, we prove that the probabilistic parameter $p$ for the parking protocol involving a coin flip does not affect the probability that a preference vector $\alpha \in [n]^n$ is a parking function. 
We also explore the properties of $\alpha$ given that it is a parking function and discuss situations in which the effect of the probabilistic parameter $p$ comes up. 
Of special interest is when $p=1/2$, where we demonstrate a sharp transition in some parking statistics. In Section~\ref{section:relatedresults}, we present several interesting combinatorial consequences of the parking protocol. 
In particular, we derive an integer sequence which solves an open problem proposed by Novelli and Thibon 2020~\cite{novelli2020duplicial}. We conclude with Section~\ref{sec:futurework}, where we offer directions for future work.

\section{Introducing probability} \label{sec:prob-intro}

We begin by stating an important result credited to Henry O.~Pollak\footnote{While there is no official paper written by Pollak on the topic, he has been credited for the proof by multiple authors, including Foata and Riordan \cite{foata}.}. We emphasize its importance in motivating the proof of Theorem~\ref{theorem:pollakcoin}, but direct the interested reader to \cite{Kimberlyy,Kimberly} for a detailed proof of counting the number of parking functions of length $n$. 
\begin{theorem}\label{theorem:pollak1}
If $n\geq 1$, then $|\PF{n}|=(n+1)^{n-1}$.
\end{theorem}

\subsection{The Coin Problem} \label{subsec:coin-prob}

We now introduce the Coin Problem, which involves a new parking protocol, motivated by Carlson et al.~in \cite{carlson2020parking}. We have $n$ cars in a queue to enter a one-way street with $n$ parking spots numbered from 1 to $n$. Let $a_i$ denote the preference of car $i$, and let $\alpha=(a_1,a_2,\ldots,a_n) \in [n]^n$ denote the vector of all $n$ cars' parking preferences. In the Coin Problem, each parking preference $a_i \in [n]$ is chosen uniformly at random, i.e., with probability $1/n$.

As in the classical parking process, each car enters the street from their preferred spot and directly parks there if possible. However, if a car's preferred spot is occupied, that driver tosses a weighted coin with probability $p$ of landing heads. If the coin toss is heads, the driver proceeds to the end of the street, either parking at the next available spot, or exiting the street if none are available. But, if the coin toss is tails, the driver proceeds in reverse, toward the start of the street. Notice that $p = 1$ in the classical process, so it may be viewed as a special case of the Coin Problem. We henceforth refer to $p$ as the ``forward probability."

Then, given this new parking protocol, for a one-way street with $n$ spots, what is the probability that a preference vector $\alpha \in \PF{n}$, chosen uniformly at random, allows all $n$ cars to park on the street? Recall that under the classical parking protocol, such an answer is easy to obtain via the pigeonhole principle: $\alpha \in [n]^n$ allows all $n$ cars to park if and only if
\begin{equation}\label{pigeon}
\#\{k: a_k \leq i\} \geq i, \hspace{.2cm} \mbox{ for all } i=1, \dots, n,
\end{equation}
and so in particular the probability of a preference vector being a parking function is invariant under the action of the symmetric group $\Sym_n$ by permuting cars. This nice and simple permutation symmetry however breaks under the probabilistic parking protocol, adding more delicacy to the answer, as we explain in the example below.
\begin{example}[Lack of permutation symmetry under the probabilistic parking protocol]\label{ex:3cars}
The preference vector $\alpha= (1,2,2) \in [3]^3$ is a parking function, or {\it parks}, with probability $p$. Cars numbered $1$ and $2$ will directly park at their preferred spot, whereas car $3$ can only park at spot $3$ as it is the only remaining available spot. This can happen only when the biased coin flips to heads, which occurs with probability $p$. The reader might find it instructive to check Figure \ref{fig:scheme} for an illustration of the above key parking steps.
\begin{figure}[H]
    \centering
    \includegraphics[width=0.65\linewidth]{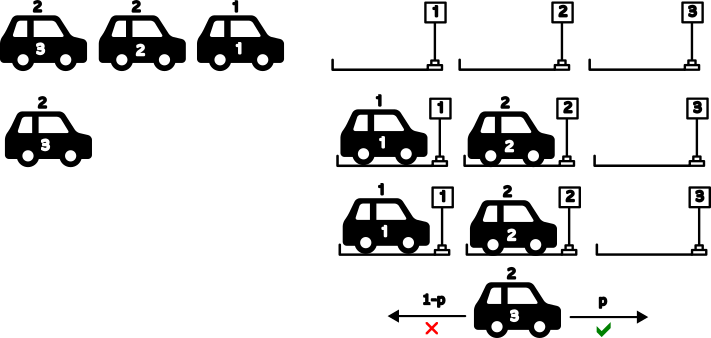}
    \caption{Parking scheme for the preference vector $(1,2,2)$. The number at the top of the car represents its preference. With probability $p$, all cars can park.}
    \label{fig:scheme}
\end{figure}
On the other hand, consider now the preference vector  $\alpha'=(2,2,1)$, which is a permutation of $\alpha=(1, 2, 2)$. Car $1$ will park at spot $2$ directly. But since the desired spot for car $2$ is now taken, instead of directly parking at spot $2$, car $2$ will have two choices. Either car $2$ moves forward to park at spot $3$, leaving spot $1$ open for car $3$, which happens to be car $3$'s desired spot. This parking situation happens with probability $p$. Or car $2$ moves backwards to park at spot $1$, which coincides with the desired spot for car $3$, forcing car $3$ to move forward to park when it enters the street. This parking situation happens with probability $(1-p)p$. Altogether, $\alpha'$ is a parking function with probability $p+(1-p)p$.
\end{example}
Given the preference vector $\alpha\in [3]^3$, Table \ref{table:probs} gives the probability that all three cars park under the probabilistic parking protocol. 
We note that each preference vector $\alpha \in [3]^3$ has probability $1/3^3=1/27$ of being chosen, and the sum of the probabilities in the table add up to $(3+1)^{3-1}=16$. This result is especially interesting because it is independent of $p$, and suggests that the probability that a random preference vector $\alpha \in \PF{n}$ is a parking function is identical to that in the classical process (see Theorem \ref{theorem:pollak1}).
\begin{table}[H]
\begin{tabular}{|c|c|c|c|c|c|c|c|c|}
\hline
 $\alpha \in [3]^3$ & $(1,1,1)$ & $(1,1,2)$  & $(1,1,3)$  & $(1,2,1)$   & $(1,2,2)$  & $(1,2,3)$ & $(1,3,1)$   \\ \hline
 $\PP(\alpha \in \PF{3} | \alpha)$ & $p^2$ & $p^2$ & $p$ & $p$ & $p$ & $1$  & $p$  \\ \hline
 $\alpha \in [3]^3$ & $(2,1,1)$ & $(2,1,2)$  & $(2,1,3)$  & $(2,2,1)$   & $(2,2,2)$  & $(2,2,3)$ & $(2,3,1)$  \\ \hline
 $\PP(\alpha \in \PF{3} | \alpha)$ & $p$ & $p$  & $1$ & $p+(1-p)p$  & $2p(1-p)$  & $p(1-p)+(1-p)$  & $1$  \\ \hline
 $\alpha \in [3]^3$ & $(3,1,1)$ & $(3,1,2)$  & $(3,1,3)$  & $(3,2,1)$   & $(3,2,2)$  & $(3,2,3)$ & $(3,3,1)$  \\ \hline
 $\PP(\alpha \in \PF{3} | \alpha)$ & $p$ & $1$ & $1-p$ & $1$ & $1-p$ & $1-p$ & $1-p$ \\ \hline
 $\alpha \in [3]^3$ & $(1,3,2)$ & $(1,3,3)$  & $(2,3,2)$  & $(2,3,3)$   & $(3,3,2)$  & $(3,3,3)$ &  \\ \hline
 $\PP(\alpha \in \PF{3} | \alpha)$ & $1$ & $1-p$ & $1-p$ & $1-p$ & $(1-p)^2$ & $(1-p)^2$ & \\ \hline
\end{tabular}
\caption{All preference vectors for $n=3$ and their probability of being a parking function.}
\label{table:probs}
\end{table}

\subsection{Invariance to the forward probability} \label{subsection:probabilityandparkingfunctions}

In this section, we generalize the results of Example~\ref{ex:3cars} and Table~\ref{table:probs}, proving that the probability that a random preference vector is a parking function is identical for all forward probabilities $p \in [0,1]$.

\begin{customthm}{1}\label{theorem:pollakcoin}
Consider the preference vector $\alpha \in [n]^n$ chosen uniformly at random. Then
\begin{eqnarray}\label{eqn:pollakcoin}
\PP(\alpha \in \PF{n} | \alpha \in [n]^n) = \frac{(n+1)^{n-1}}{n^n}.
\end{eqnarray}
\end{customthm}

In particular, despite the forward probability $p$, the probability that a random preference vector in $[n]^n$ is a parking function is identical to the classical case, in which cars may only search in the forward direction.

Similar to Pollak's proof for counting classical parking functions, as cited in Theorem \ref{theorem:pollak1}, we consider the scenario in which $n$ cars park on a circle with spots $[n+1]$ arranged clockwise, where each car may prefer any spot in $[n+1]$. Then, there are $(n+1)^{n}$ equally likely preference vectors. 
Our probabilistic parking protocol for one-way parking may be interpreted similarly in this circular parking situation. 
If a car arrives at its preferred spot and finds it unoccupied it parks there. If instead the spot is occupied, then the driver tosses the biased coin. If the coin lands on heads, with probability $p$, the driver continues moving clockwise in the street. 
However, if the coin lands on tails, with probability $1-p$, the car moves counterclockwise and tries to find an unoccupied parking spot.

For convenience, we define the following set.
\begin{definition}\label{def:bucket}
For all $i \in [n+1]$, let $S_i = \{\alpha \in [n+1]^n : \text{spot \(i\) is vacant after all \(n\) cars have parked}\}.$
\end{definition}


In the classical case, any preference vector in $S_{n+1}$ must also be a parking function. This is because the forward probability is $p = 1$, so cars may only search in the forward direction. Conversely, any preference vector which is a parking function must belong in $S_{n+1}$, because no car may be assigned to spot $(n+1)$. However, this relationship between $S_{n+1}$ and $\PF{n}$ is not limited to the classical case.

\begin{lemma}\label{lemma:does-not-contain}
Given arbitrary preference vector $\alpha \in [n+1]^n$, then $\alpha \in \PF{n}$ if and only if $\alpha \in S_{n+1}$.
\end{lemma}

\begin{proof}
In the forward direction, if $\alpha \in \PF{n}$, then none of the $n$ cars park in spot $(n+1)$, which implies that $\alpha \in S_{n+1}.$

In the backward direction, if $\alpha \in S_{n+1}$, then none of the $n$ cars prefer spot $(n+1)$. For contradiction, suppose some car $i$ preferred spot $(n+1)$. Then car $i$ either parks in spot $(n+1)$ directly, or it cannot, because that spot is already occupied. In either case, some car parks in spot $(n+1)$, and therefore $\alpha \not \in S_{n+1}$. That is, all $n$ cars prefers some spot from the original one-way street. And, because $\alpha \in S_{n+1}$, all cars are able to fit on the one-way street. Therefore, $\alpha \in \PF{n}$.
\end{proof}

Next, recall that in Table~\ref{table:probs}, we saw that
$p$ did not affect the final probability $\PP(\alpha \in \PF{3} | \alpha \in [3]^3)$.
This was surprising because $p$ appeared in the probabilities corresponding to
all cars parking on the street, as shown in Example~\ref{ex:3cars}. This motivates analysis via conditional probability.


\begin{definition}\label{def:pmf-coin}
Given an arbitrary preference vector $\alpha \in [n+1]^n$, let the conditional probability mass function $f_{\alpha}: [n+1] \rightarrow [0, 1]$ denote the probability that $\alpha \in S_i$. That is,
$f_\alpha(i) = \PP(\alpha \in S_i | \alpha)$.
\end{definition}

In Definition \ref{def:bucket}, $S_i$ denotes the event that all $n$ cars have parked on the circular street with $n+1$ spots, only leaving spot $i$ vacant. 
The conditional sample space for some $\alpha \in [n+1]^n$ is thus partitioned into the disjoint events $\{\alpha \in S_1, \alpha \in S_2, \dots, \alpha \in S_{n+1}\}$. 
This implies that $\sum_{i=1}^{n+1} f_\alpha(i)=1$ by Definition \ref{def:pmf-coin}.

\begin{lemma}\label{lemma:pmf-modular}
Consider any two preference vectors $\alpha, \alpha^* \in [n+1]^n$ satisfying $\alpha^* = \alpha + (k * {\bf 1})\ (\textrm{mod}\ (n+1))$, for some $k \in [n]$. Let $j = (i + k)\ (\textrm{mod}\ (n+1))$. Then $f_\pi(i) = f_{\pi^*}(j)$.
\end{lemma}

\begin{proof}
This follows from the fact that the preference vectors themselves are congruent modulo $(n+1)$. In particular, $\alpha^*$ is derived by adding $k \ (\textrm{mod}\ (n+1))$ to each preference in $\alpha$, so the probability that $\alpha \in S_i$ is equivalent to the probability that $\alpha^* \in S_j$.
\end{proof}

\begin{definition}\label{def:leading-pref}
Let $L_a$ denote the set of all preference vectors which have leading preference $a \in [n+1]$. That is, $L_a = \{\alpha \in [n+1]^n : a_1 = a\}.$
\end{definition}

\begin{remark}\label{remark:leading-size}
Notice that $|L_a| = (n+1)^n / (n+1) = (n+1)^{n-1}$, where we divide by $(n+1)$ because the first preference is fixed to be $a$.
\end{remark}


\begin{definition}\label{def:cdf-coin}
For all cars $i \in [n+1]$ and leading preferences $a \in [n+1]$, let $F_a \in [0, n+1]^{n+1}$ denote a length-$(n+1)$ vector satisfying
\begin{equation}
F_a^{(i)} = \sum\limits_{\forall \alpha \in {L_a}} f_\alpha(i).  
\end{equation}
\end{definition}
\
Intuitively, if we simulated the parking procedure for all preference vectors $\alpha \in L_a$, then $F_a^{(i)}$ denotes the expected number of preference vectors which have leading preference $a$ and leave spot $i$ vacant. Notice that $F_{a}^{(a)}=0$, because all $\alpha \in L_a$ satisfy $f_\alpha(a) = 0$. i.e., spot $a$ will always be occupied by the leading car with preference $a$. Applying Definition~\ref{def:cdf-coin}, it follows from Lemma~\ref{lemma:pmf-modular} that $F_a^{(i)} = F_{b}^{(j)}$, for all $a,b,i,j \in [n+1]$ satisfying $b - a \equiv j - i \ (\textrm{mod}\ (n+1))$. 


Next, purely as a visual aid, we construct a $(n+1) \times (n+1)$ matrix ${\bf M}$, where entry ${\bf M}_{a,i} = F_a^{(i)}$. Then ${\bf M}_{a,i} = {\bf M}_{a+1, i+1} = \cdots = {\bf M}_{a+n, i+n}$, where the subscripts are modulo $(n+1)$. In other words, each row in ${\bf M}$ is the same as the previous row, but with entries shifted one cell to the right, modulo $(n+1)$. That is,
\begin{eqnarray}
{\bf M}=
\begin{pmatrix}
F_1^{(1)} & F_1^{(2)} & \dots & F_1^{(n+1)}\\
F_2^{(1)} & F_2^{(2)} & \dots & F_2^{(n+1)}\\
\vdots & \vdots & \ddots & \vdots\\
F_{n+1}^{(1)} & F_{n+1}^{(2)} & \dots & F_{n+1}^{(n+1)}\\
\end{pmatrix}
=
\begin{pmatrix}
F_1^{(1)} & F_1^{(2)} & \dots & F_1^{(n+1)}\\
F_1^{(n+1)} & F_1^{(1)} & \dots & F_1^{(n)}\\
\vdots & \vdots & \ddots & \vdots\\
F_{1}^{(2)} & F_{1}^{(3)} & \dots & F_{1}^{(1)}\\
\end{pmatrix}.
\label{eqn:matrix-visual}
\end{eqnarray}

Equation~\eqref{eqn:matrix-visual} highlights the underlying symmetry via partitioning preference vectors based on leading preference. Note that while the simplified result is written in terms of entries from the first row, i.e., with subscript $a = 1$, similar equivalence holds for any leading preference $a \in [n+1]$. In particular, every column and row has the same sum. That is,
\begin{eqnarray}
\sum\limits_{a=1}^{n+1} F_a^{(1)} = \cdots = \sum\limits_{a=1}^{n+1} F_a^{(n+1)} = \sum\limits_{i=1}^{n+1} F_1^{(i)} = \cdots = \sum\limits_{i=1}^{n+1} F_{n+1}^{(i)}.
\label{eqn:equal-column-row}
\end{eqnarray}

This symmetry helps us compute the following probability:
\begin{equation}
\begin{aligned}
\PP(\alpha \in S_{n+1})
&= \sum_{\forall \alpha} \PP(\alpha \in S_{n+1} | \alpha)\ \PP(\alpha) \quad&&\text {(By law of total probability)}\\
&= \sum\limits_{a=1}^{n+1} \sum\limits_{\forall \alpha \in L_a} \PP(\alpha \in S_{n+1} | \alpha)\frac{1}{(n+1)^n}.
\end{aligned}
\end{equation}

Notice that we are partitioning the sample space of all $\alpha \in [n+1]^n$ into the disjoint sets $L_1, L_2, \dots, L_{n+1}$. And, since each $\alpha$ is chosen uniformly at random, we have that $\PP(\alpha) = 1/(n+1)^n$. We now complete the computation:
\begin{equation}
\begin{aligned}
\PP(\alpha \in S_{n+1})
&= \frac{1}{(n+1)^n}\sum\limits_{a=1}^{n+1} \sum\limits_{\forall \alpha \in L_a} f_\alpha(n+1) \quad&&\text{(By Definition \(\ref{def:pmf-coin}\))} \\
&= \frac{1}{(n+1)^n}\sum\limits_{a=1}^{n+1} F_a^{(n+1)} \quad&&\text{(By Definition \(\ref{def:cdf-coin}\))} \\
&= \frac{1}{(n+1)^n} \sum\limits_{i=1}^{n+1} F_1^{(i)} \quad&&\text{(By Equation~\eqref{eqn:equal-column-row})} \\
&= \frac{1}{(n+1)^n} \sum\limits_{i=1}^{n+1} \sum\limits_{\forall \alpha \in {L_1}} f_\alpha(i) \quad&&\textrm{(By Definition \(\ref{def:cdf-coin}\))} \\
&= \frac{1}{(n+1)^n} \sum\limits_{\forall \alpha \in {L_1}} \sum\limits_{i=1}^{n+1} f_\alpha(i) \\
&= \frac{1}{(n+1)^n} \sum\limits_{\forall \alpha \in {L_1}} 1 \quad&&\textrm{(by Definition \(\ref{def:pmf-coin}\))} \\
&= \frac{(n+1)^{n-1}}{(n+1)^n}. \quad&&\textrm{(by Remark~\(\ref{remark:leading-size}\))}
\end{aligned}
\end{equation}

Recall that the relevant sample space for Theorem~\ref{theorem:pollakcoin} is $[n]^n$, and not $[n+1]^n$. That is, we want to find $\PP(\alpha \in S_{n+1} | \alpha \in [n]^n)$. Since this is a conditional probability, it is useful to consider the intersection of the two events. From Lemma~\ref{lemma:does-not-contain}, we know that $\alpha \in S_{n+1}$ implies that $\alpha \in [n]^n$. However, the converse is not necessarily true. As a counterexample, suppose all $n$ cars prefer spot~1. 
Then, car 1 parks at spot 1, and car 2 may park at spot $(n+1)$ with probability $1-p$, since its preferred spot is occupied. Therefore, $\PP(\alpha \in S_{n+1} \cap \alpha \in [n]^n) = \PP(S_{n+1})$.

Now, we may complete the derivation of Theorem~\ref{theorem:pollakcoin}:
\begin{equation}
\begin{aligned}
\PP(\alpha \in S_{n+1} | \alpha \in [n]^n) 
&= \frac{\PP(\alpha \in S_{n+1} \cap \alpha \in [n]^n)}{\PP(\alpha \in [n]^n)}\\
&= \frac{\PP(\alpha \in S_{n+1})}{\PP(\alpha \in [n]^n)}\\
&= \frac{(n+1)^{n-1}/(n+1)^n}{n^n/(n+1)^{n}}\\
&= \frac{(n+1)^{n-1}}{n^n}.
\end{aligned}
\end{equation}

Therefore, the probability that an arbitrary preference vector in $[n]^n$ assigns all $n$ cars to a unique spot on a one-way street with $n$ spots is completely independent of the forward probability $p$.

{
\subsection{Effect of the forward probability}\label{sec:impactppin}
In the previous section, we showed  that the probability that a random preference vector $\alpha \in [n]^n$ is a parking function does not depend on the forward probability $p$. 
Thus, a natural question to ask is: what kind of parking statistics does depend on $p$? In this section, we present one such parking statistic and explore its statistical properties.

Our investigations in this section rely on a combinatorial construction which we term a \textit{parking function shuffle} and Abel's extension of the binomial theorem. 
These concepts were first discussed in Diaconis and Hicks \cite{diaconis} and later extended further in Kenyon and Yin \cite{KY}. 
Some asymptotic expansion formulas will also prove useful. We provide some background on these concepts first.

\begin{definition}\label{shuffle}
Let $1\leq k \leq n$. Say that $(a_1, \dots, a_{n-1})$ is a parking function shuffle of the parking function $\alpha \in \PF{k-1}$ and $\beta \in \PF{n-k}$ if $a_1, \dots, a_{n-1}$ is any permutation of the union of the two words $\alpha$ and $\beta+(k, \dots, k)$.
\end{definition}

\begin{example}
Take $n=8$ and $k=4$. Take $\alpha=(2, 1, 2) \in \PF{3}$ and $\beta=(1, 2, 4, 3) \in \PF{4}$. Then $(2, \underline{5},2, \underline{8}, \underline{7},1, \underline{6})$ is a shuffle of the two words $(2, 1, 2)$ and $(5, 6, 8, 7)$.
\end{example}

\begin{theorem}[Abel's extension of the binomial theorem, derived from Pitman \cite{Pitman} and Riordan \cite{Riordan}]\label{Abel}
Let
\begin{equation}\label{b}
A_n(x, y; p, q)=\sum_{s=0}^n \binom{n}{s} (x+s)^{s+p} (y+n-s)^{n-s+q}.
\end{equation}
Then
\begin{align}\label{b1}
A_n(x, y; p, q)&=A_n(y, x; q, p),\\
\label{b2}
A_n(x, y; p, q)&=A_{n-1}(x, y+1; p, q+1)+A_{n-1}(x+1, y; p+1, q),\\\label{b3}
A_n(x, y; p, q)&=\sum_{s=0}^{n} \binom{n}{s}s!(x+s)A_{n-s}(x+s, y; p-1, q).
\end{align}
Moreover, the following special instances hold via the basic recurrences listed above:
\begin{align}\label{1}
&A_n(x, y; -1, -1)=(x^{-1}+y^{-1})(x+y+n)^{n-1},
\\\label{2}
&A_n(x, y; -1, 0)=x^{-1}(x+y+n)^n,
\\
\label{3}
&A_n(x, y; -1, 1)=x^{-1} \sum_{s=0}^n \binom{n}{s} (x+y+n)^s (y+n-s) (n-s)!,
\\
\label{4}
&A_n(x, y; 0, 0)=\sum_{s=0}^n \binom{n}{s} (x+y+n)^s (n-s)!.
\end{align}
\end{theorem}

\begin{lemma}\label{CLT}
Let $X_1, X_2, \dots$ be iid Poisson$(1)$ random variables. Then
\begin{equation}
\PP(X_1+\cdots+X_n \leq n)=\frac{1}{2}+\frac{2}{3}\frac{1}{\sqrt{2\pi n}}+o\left(\frac{1}{\sqrt{n}}\right).
\end{equation}
\end{lemma}

\begin{proof}
We apply the continuity-corrected Edgeworth expansion as in Esseen \cite{Esseen} and Kolassa and McCullagh \cite{KM}, so that
\begin{multline}
\PP(X_1+\cdots+X_n \leq n+x\sqrt{n})=\PP\left(\frac{X_1+\cdots+X_n-n}{\sqrt{n}}\leq x\right)\\=\Phi(0)+\frac{\exp(-x^2/2)}{\sqrt{2\pi n}}\left(\frac{\mu_3}{6\sigma^3}(1-x^2)+\frac{1}{\sigma}D(x\sigma\sqrt{n})\right)+o\left(\frac{1}{\sqrt{n}}\right),
\end{multline}
where $\sigma$ is the standard deviation and $\mu_3$ is the third central moment of $X_1$, $\Phi(x)$ is the distribution function of the standard normal, and $D(x)=\lfloor x \rfloor-x+1/2$
is a periodic function that addresses the discontinuity in the distribution function of the discrete random variable $(X_1+\cdots+X_n-n)/\sqrt{n}$.
\end{proof}

We are now ready to show that the distribution of $a_n$, the parking preference of the last car, depends on $p$ and give explicit formulas.

\begin{customthm}{3}\label{component}
Consider the preference vector $\alpha \in [n]^n$, chosen uniformly at random. Then given that $\alpha \in \PF{n}$,
\begin{align}
&\PP(a_n=j|\alpha \in \PF{n})=\frac{2}{n+1}-\frac{1}{(n+1)^{n-1}} \Big[p\sum_{s=n-j+1}^{n-1} \binom{n-1}{s} (n-s)^{n-s-2} (s+1)^{s-1}\notag \\
&\hspace{5cm}+(1-p) \sum_{s=0}^{n-j-1} \binom{n-1}{s} (n-s)^{n-s-2} (s+1)^{s-1}\Big],
\end{align}
where $a_n$ denotes the parking preference of the last car.
\end{customthm}

\begin{remark}\label{rmk:symmetry}
Note the parking symmetry as observed in the introduction: $\PP(a_n=j|\alpha \in \PF{n})$ under protocol with parameter $p$ equals $\PP(a_n=n+1-j|\alpha \in \PF{n})$ under protocol with parameter $1-p$. See Figure \ref{distribution} and Table \ref{table:probs}.
\end{remark}

\begin{proof}[Proof of Theorem \ref{component}]
Cars $1, \dots, n-1$ have all parked along the one-way street before the $n$th car enters, leaving only one open spot $k$ for the $n$th car to park. Since a car cannot jump over an empty spot, the parking protocol implies that $(a_1, \dots, a_{n-1})$ is a parking function shuffle of $\alpha \in \PF{k-1}$ and $\beta \in \PF{n-k}$, and $\alpha$ and $\beta$ do not interact with each other. This open spot $k$ could be either the same as $j$, the preference of the last car, in which case the car parks directly. Or, $k$ could be bigger than or less than $j$, in which case whether the last car parks or not depends on the outcome of the biased coin flip, as it will dictate the car to go forward or backward. Using Theorem \ref{theorem:pollakcoin}, we have $\PP(a_n=j|\alpha \in \PF{n})$
\begin{align}
&=\frac{1}{(n+1)^{n-1}}\Big[(1-p)\sum_{k=1}^{j-1} \binom{n-1}{n-k} k^{k-2} (n-k+1)^{n-k-1}+\binom{n-1}{n-j} j^{j-2} (n-j+1)^{n-j-1} \notag \\
&\hspace{5cm} +p\sum_{k=j+1}^{n} \binom{n-1}{n-k} k^{k-2} (n-k+1)^{n-k-1}\Big]\\
&=\frac{1}{(n+1)^{n-1}} \Big[\sum_{k=1}^{n} \binom{n-1}{n-k} k^{k-2} (n-k+1)^{n-k-1}-p\sum_{k=1}^{j-1} \binom{n-1}{n-k} k^{k-2} (n-k+1)^{n-k-1} \notag \\
&\hspace{5cm} -(1-p)\sum_{k=j+1}^{n} \binom{n-1}{n-k} k^{k-2} (n-k+1)^{n-k-1}\Big]\label{eq:kk}\\
&=\frac{1}{(n+1)^{n-1}} \Big[\sum_{s=0}^{n-1} \binom{n-1}{s} (n-s)^{n-s-2} (s+1)^{s-1}-p\sum_{s=n-j+1}^{n-1} \binom{n-1}{s} (n-s)^{n-s-2} (s+1)^{s-1} \notag \\
&\hspace{5cm} -(1-p) \sum_{s=0}^{n-j-1} \binom{n-1}{s} (n-s)^{n-s-2} (s+1)^{s-1} \Big] \label{eq:ss}\\
&=\frac{2}{n+1}-\frac{1}{(n+1)^{n-1}} \Big[p\sum_{s=n-j+1}^{n-1} \binom{n-1}{s} (n-s)^{n-s-2} (s+1)^{s-1}\notag \\
&\hspace{5cm}+(1-p) \sum_{s=0}^{n-j-1} \binom{n-1}{s} (n-s)^{n-s-2} (s+1)^{s-1}\Big] \label{eq: abel}.
\end{align}
Here the binomial coefficient $\binom{n-1}{n-k}$ accounts for the shuffling of the length $k-1$ subsequence and the length $n-k$ subsequence of parking preferences among the $1$st through the $(n-1)$st cars. We used a simple change of variables $s=n-k$ going from (\ref{eq:kk}) to \ref{eq:ss}. From (\ref{eq:ss}) to (\ref{eq: abel}), we applied Abel's binomial identity (\ref{1}): $A_{n-1}(1, 1; -1, -1)=2(n+1)^{n-2}$.
\end{proof}

\begin{customthm}{4}\label{mean}
Take $n$ large. For preference vector $\alpha \in [n]^n$ chosen uniformly at random, we have
\begin{equation}
\mathbb{E}(a_n | \alpha \in \PF{n})=\frac{n+1}{2}-(2p-1)\Big[\frac{\sqrt{2\pi}}{4}n^{1/2}-\frac{7}{6}\Big]+o(1).
\end{equation}
\end{customthm}

\begin{remark}
Note the parking symmetry as observed in the introduction: the sum of $\mathbb{E}(a_n | \alpha \in \PF{n})$ under protocol with parameter $p$ and $\mathbb{E}(a_n | \alpha \in \PF{n})$ under protocol with parameter $1-p$ is $n+1$. The lower order correction terms from $(n+1)/2$ for $\mathbb{E}(a_n | \alpha \in \PF{n})$ vanish completely under protocol with parameter $p=1/2$, and $\mathbb{E}(a_n | \alpha \in \PF{n})=(n+1)/2$ exactly. See Figure \ref{fig:mean}.
\end{remark}

\begin{figure}[htp]
\centering
\includegraphics[width=4in, trim= 0 0cm 0cm 2cm, clip]{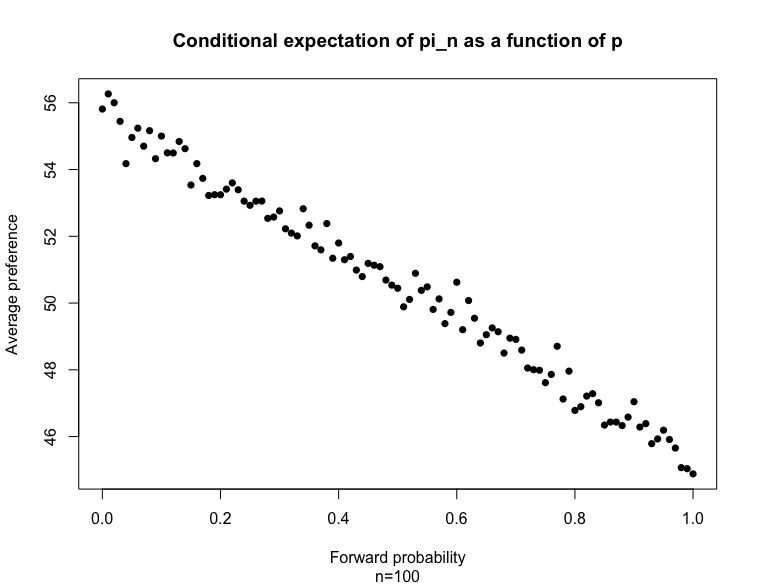}
\caption{The conditional expectation of $a_n$ (parking preference of the last car) given $\alpha \in \PF{n}$ under the probabilistic parking protocol with parameter $p$. Here size $n=100$ and $100,000$ samples are drawn uniformly at random.}\label{fig:mean}
\end{figure}

\begin{proof}[Proof of Theorem \ref{mean}]
By Theorem \ref{component},
\begin{align}
&\mathbb{E}(a_n | \alpha \in \PF{n})=\sum_{j=1}^n j\PP(a_n=j|\alpha \in \PF{n}) \notag \\
&=n-\frac{1}{(n+1)^{n-1}} \sum_{j=1}^n j\Big[p\sum_{s=n-j+1}^{n-1} \binom{n-1}{s} (n-s)^{n-s-2} (s+1)^{s-1}\notag \\
&\hspace{5cm}+(1-p) \sum_{s=0}^{n-j-1} \binom{n-1}{s} (n-s)^{n-s-2} (s+1)^{s-1}\Big] \notag \\
&=n-\frac{1}{(n+1)^{n-1}} \sum_{s=0}^{n-1} \binom{n-1}{s} (n-s)^{n-s-2} (s+1)^{s-1} \Big[p\sum_{j=n-s+1}^n j+(1-p)\sum_{j=1}^{n-s-1} j \Big] \notag \\
&=n-\frac{1}{2(n+1)^{n-1}} \Bigg[\sum_{s=0}^{n-1} \binom{n-1}{s} (n-s)^{n-s-2} (s+1)^{s-1} \cdot \notag \\
&\hspace{4cm} \cdot [(1-p)(n-s)^2+p(n-s)(s+1)+p(n+1)(s+1)-(n-s)-p(n+1)]\Bigg] \label{A1}\\
&=n-\frac{1}{2(n+1)^{n-1}} \Big[(1-p)A_{n-1}(1, 1; -1, 1)+pA_{n-1}(1, 1; 0, 0)\notag \\
&\hspace{4cm}+(pn+p-1)A_{n-1}(1, 1; -1, 0)-p(n+1)A_{n-1}(1, 1; -1, -1) \Big], \label{A2}
\end{align}
where Abel's binomial theorem is used multiple times from \eqref{A1} to \eqref{A2}.

We proceed to estimate \eqref{A2} asymptotically. Using \eqref{1}, $A_{n-1}(1, 1; -1, -1)=2(n+1)^{n-2}$ and using \eqref{2}, $A_{n-1}(1, 1; -1, 0)=(n+1)^{n-1}$. By \eqref{3} and \eqref{4},
\begin{align}
&(1-p)A_{n-1}(1, 1; -1, 1)+pA_{n-1}(1, 1; 0, 0) \notag \\
&=(1-p) \sum_{s=0}^{n-1} \binom{n-1}{s} (n+1)^s (n-s) (n-1-s)! +p \sum_{s=0}^{n-1} \binom{n-1}{s} (n+1)^s (n-1-s)! \notag \\
&=(n-np+p)(n+1)^{n-1}+(2p-1)(n-1)! \sum_{s=0}^{n-2} \frac{(n+1)^s}{s!}.
\end{align}
From Stirling's formula,
\begin{align}
(n-1)! \sim \sqrt{2\pi (n-1)}e^{-(n-1)} (n-1)^{n-1} \left[1+\frac{1}{12(n-1)}\right].
\end{align}
We also recognize that
\begin{equation}
e^{-(n+1)} \sum_{s=0}^{n-2} \frac{(n+1)^s}{s!}
\end{equation}
equals the probability that the sum of $n+1$ iid Poisson$(1)$ random variables is less than or equal to $n-2$, and is asymptotic to
\begin{equation}
\frac{1}{2}+\frac{2}{3}\frac{1}{\sqrt{2\pi(n+1)}}-3e^{-(n+1)} \frac{(n+1)^{n-1}}{(n-1)!}
\end{equation}
from the Edgeworth expansion in Lemma \ref{CLT}. Combining the above,
\begin{align}
\mathbb{E}(a_n | \alpha \in \PF{n})&=\frac{n+1}{2}-(2p-1)\frac{(n-1)!}{2(n+1)^{n-1}} \sum_{s=0}^{n-2} \frac{(n+1)^s}{s!} \notag\\
&=\frac{n+1}{2}-(2p-1)\Big[\frac{\sqrt{2\pi}}{4}n^{1/2}-\frac{7}{6}\Big]+o(1).
\end{align}
\end{proof}

\subsection{Why the forward probability $p=1/2$ is special}\label{sec:rate}
The conditional distribution of $a_n$ for the classical model is close to a uniform distribution over $[n]$ in a precise sense given by the total variation distance. 
Building upon the results in the last section, in this section we show that the presence of $p$ has a consequential impact on this rate of convergence to uniform. 
Specifically, for $p=1/2$ the symmetry provided by $p$ plays an important role of speeding up the rate of convergence. 
This explains why in the histogram of Figure \ref{distribution} for $p=1/2$ we already see a distribution that looks uniform, whereas for other values of $p$ the presence of more mass on small or large values of $[n]$ is quite apparent.

In order to better discuss our results in this section, let us introduce a few definitions and set our notation. For two distributions $P$ and $Q$ over $[n]$, their total variation distance is given by
\begin{equation}\label{def:dTV}
    \| P - Q \|_{\TV} := \frac{1}{2}\sum_{j=1}^n| P(j) - Q(j) |.
\end{equation}
For $p\in[0,1]$, for notational convenience, let $Q_{n,p}$ be the conditional probability of $a_n$ given $\alpha \in PF_n$ under the probabilistic parking protocol with parameter $p$. That is, $Q_{n,p}$ is the distribution over $[n]$ established in Theorem \ref{component}: $Q_{n, p}(j)=\PP(a_n=j|\alpha \in \PF{n})$. Additionally, let $\unif{n}$ be the uniform distribution over $[n]$: $\unif{n}(j)=1/n$ for all $j \in [n]$.

In \cite{diaconis} Diaconis and Hicks showed that $\| Q_{n,1}- \unif{n} \|_{\TV}$ goes to zero as $n$ goes to infinity. They also conjectured that the rate of convergence is $O(1/\sqrt{n})$. In \cite{bellin2021} Bellin confirmed that their conjecture is indeed true.
\begin{customthm}{5}\label{thm:conv12}
Let $Q_{n,p}(\cdot)$ be $\PP(a_n = \cdot \;|\; \alpha \in \PF{n})$ under the probabilistic parking protocol with parameter $p$. For $p=1/2$, the following bound holds
    \begin{equation}
     \| Q_{n,1/2}- \unif{n} \|_{\TV} = \Theta\left(\frac{1}{n}\right).
    \end{equation}
\end{customthm}
\begin{proof} We use the definition of total variation distance together with Theorem \ref{component} as follows. For $j \in [n]$, Theorem \ref{component} gives us
\begin{align}
    &Q_{n,1/2}(j)=\frac{2}{n+1}-\frac{1}{(n+1)^{n-1}} \Big[\frac{1}{2}\sum_{s=n-j+1}^{n-1} \binom{n-1}{s} (n-s)^{n-s-2} (s+1)^{s-1}\notag \\\label{rhs}
    &\hspace{6cm}+\frac{1}{2} \sum_{s=0}^{n-j-1} \binom{n-1}{s} (n-s)^{n-s-2} (s+1)^{s-1}\Big].
\end{align}
Adding and subtracting $\binom{n-1}{n-j}j^{j-2}(n-j+1)^{n-j-1}/2(n+1)^{n-1}$ on the right-hand side of \eqref{rhs}, and recalling that $A_{n-1}(1,1;-1,-1) = 2(n+1)^{n-2}$ yields
\begin{align}
        Q_{n,1/2}(j) & = \frac{2}{n+1} - \frac{1}{(n+1)^{n-1}}\left[\frac{1}{2}A_{n-1}(1,1;-1,-1)-\frac{1}{2}\binom{n-1}{n-j}j^{j-2}(n-j+1)^{n-j-1} \right] \notag \\\label{eq:identity}
        & = \frac{1}{n+1} + \frac{\binom{n-1}{n-j}j^{j-2}(n-j+1)^{n-j-1}}{2(n+1)^{n-1}}.
\end{align}
Thus, substituting the identity in \eqref{eq:identity} into \eqref{def:dTV}, we obtain
\begin{equation}\label{eq:dTV}
    \begin{split}
        \| Q_{n,1/2}- \unif{n} \|_{\TV} & = \frac{1}{2}\sum_{j=1}^n\left | \frac{1}{n+1} + \frac{\binom{n-1}{n-j}j^{j-2}(n-j+1)^{n-j-1}}{2(n+1)^{n-1}} - \frac{1}{n}\right | \\
        & = \frac{1}{2}\sum_{j=1}^n\left | \frac{\binom{n-1}{n-j}j^{j-2}(n-j+1)^{n-j-1}}{2(n+1)^{n-1}} - \frac{1}{n(n+1)}\right | \\
        & = \frac{1}{2}\sum_{s=0}^{n-1}\left | \frac{\binom{n-1}{s}(n-s)^{n-s-2}(s+1)^{s-1}}{2(n+1)^{n-1}} - \frac{1}{n(n+1)}\right | \\
        & \le \frac{1}{2}\sum_{s=0}^{n-1}\left[\frac{\binom{n-1}{s}(n-s)^{n-s-2}(s+1)^{s-1}}{2(n+1)^{n-1}} + \frac{1}{n(n+1)}\right] \\
        & = \frac{A_{n-1}(1,1;-1,-1)}{4(n+1)^{n-1}}+\frac{1}{2(n+1)} = \frac{1}{n+1} \sim \frac{1}{n}.
    \end{split}
\end{equation}
To obtain a lower bound of the same order, notice that 
\begin{equation}\label{eq:dTV}
    \begin{split}
        \| Q_{n,1/2}- \unif{n} \|_{\TV} & = \frac{1}{2}\sum_{s=0}^{n-1}\left | \frac{\binom{n-1}{s}(n-s)^{n-s-2}(s+1)^{s-1}}{2(n+1)^{n-1}} - \frac{1}{n(n+1)}\right | \\
        & \ge \frac{1}{2}\left | \frac{\binom{n-1}{n-1}(1)^{-1}n^{n-2}}{2(n+1)^{n-1}} - \frac{1}{n(n+1)}\right |\\ 
        & \ge \frac{n^{n-2}}{4(n+1)^{n-1}} - \frac{1}{n(n+1)} \sim \frac{1}{4e} \frac{1}{n},
    \end{split}
\end{equation}
which proves that $\| Q_{n,1/2}- \unif{n} \|_{\TV} = \Theta(1/n)$.
\end{proof}
Next we show that indeed the case $p=1/2$ is special when considering the rate of convergence of $\| Q_{n,p}- \unif{n} \|_{\TV}$ to zero. We show that for any $p$ other than $p=1/2$, the sequence $(\| Q_{n,p}- \unif{n} \|_{\TV})_{n\ge 1}$ cannot converge to zero much faster than the classical model given by $p=1$. 
\begin{proposition}\label{prop:dTVp} For any $p\neq 1/2$, the following bound holds
\begin{equation}
|2p-1|\| Q_{n,1}- \unif{n} \|_{\TV}\le \| Q_{n,p}- \unif{n} \|_{\TV} \le \| Q_{n,1}- \unif{n} \|_{\TV}.
\end{equation}
\end{proposition}

\begin{remark}
Proposition \ref{prop:dTVp} shows that the rate of convergence is of the same order for any $p \neq 1/2$. This important fact will be used to establish the exact rate of convergence for generic $p \neq 1/2$ in Theorem \ref{thm:other}.
\end{remark}

\begin{proof}[Proof of Proposition \ref{prop:dTVp}]
Recall that by Theorem~\ref{component}, for any $p\in [0,1]$ and $n \in \mathbb{N}$, $Q_{n,p}$ is the convex combination of $Q_{n,1}$ and $Q_{n,0}$. More precisely,
\begin{equation}\label{eq:convex}
    Q_{n,p} = pQ_{n,1} + (1-p)Q_{n,0}.
\end{equation}
Moreover, since $Q_{n,1}(j) = Q_{n,0}(n+1-j)$ (see Remark~\ref{rmk:symmetry}), it follows that 
\begin{equation}\label{eq:dTV01}
\| Q_{n,1}- \unif{n} \|_{\TV} = \| Q_{n,0}- \unif{n} \|_{\TV}.
\end{equation}

Let us assume $p>1/2$ for the moment. Via the above symmetry, the reader will notice that the case for~$p<1/2$ follows a similar argument. From the definition of the total variation distance combined with \eqref{eq:convex} and that $|a+b|\ge |a|-|b|$, we obtain
\begin{equation}\label{ineq:lb}
    \begin{split}
        \| Q_{n,p}- \unif{n} \|_{\TV} & = \| pQ_{n,1} + (1-p)Q_{n,0}- \unif{n} \|_{\TV} \\
        & = \frac{1}{2}\sum_{j=1}^n\left| pQ_{n,1}(j) + (1-p)Q_{n,0}(j) - \frac{1}{n}\right|\\
        &=\frac{1}{2}\sum_{j=1}^n\left| pQ_{n,1}(j) - \frac{p}{n} + (1-p)Q_{n,0}(j) - \frac{(1-p)}{n}\right|\\
        & \ge \frac{1}{2}\sum_{j=1}^n\left| pQ_{n,1}(j) - \frac{p}{n} \right| - \frac{1}{2}\sum_{j=1}^n\left |(1-p)Q_{n,0}(j) - \frac{(1-p)}{n}\right|\\
        & = p\| Q_{n,1}- \unif{n} \|_{\TV} -(1-p)\| Q_{n,0}- \unif{n} \|_{\TV}\\
        & = (2p-1)\| Q_{n,1}- \unif{n} \|_{\TV}.
    \end{split}
\end{equation}
For the upper bound, we instead apply the triangle inequality to the third line of \eqref{ineq:lb}:
\begin{equation}\label{ineq: ub}
    \begin{split}
        \| Q_{n,p}- \unif{n} \|_{\TV} & =\frac{1}{2}\sum_{j=1}^n\left| pQ_{n,1}(j) - \frac{p}{n} + (1-p)Q_{n,0}(j) - \frac{(1-p)}{n}\right|\\
        & \le \frac{1}{2}\sum_{j=1}^n\left| pQ_{n,1}(j) - \frac{p}{n} \right| + \frac{1}{2}\sum_{j=1}^n\left |(1-p)Q_{n,0}(j) - \frac{(1-p)}{n}\right|\\
        & = \frac{1}{2} \left(p\| Q_{n,1}- \unif{n} \|_{\TV} + (1-p)\| Q_{n,0}- \unif{n} \|_{\TV}\right)\\
        & = \| Q_{n,1}- \unif{n} \|_{\TV}.\qedhere
    \end{split}
\end{equation}
\end{proof}

For the next result we use the following alternative definition for the total variation distance.
\begin{proposition}[Proposition 4.5 in Levin and Peres \cite{levin2017markov}]\label{prop:altdTV} Let $P$ and $Q$ be two distributions over $[n]$, then 
    \begin{equation}\label{test}
        \| P - Q \|_{\TV} = \frac{1}{2} \sup \left \lbrace \sum_{j=1}^n f(j)P(j) - \sum_{j=1}^nf(j)Q(j): \max_j|f(j)| \le 1 \right \rbrace.
    \end{equation}
    
\end{proposition}
With the alternative definition in hand, we can show the right order of the rate of convergence of $\| Q_{n,p}- \unif{n}\|_{\TV}$ for $p\neq 1/2$ which is our next result.
\begin{customthm}{6}\label{thm:other}
Let $Q_{n,p}(\cdot)$ be $\PP(a_n = \cdot \;|\; \alpha \in \PF{n})$ under the probabilistic parking protocol with parameter $p$. For $p\neq 1/2$, we have that 
    \begin{equation}
     \| Q_{n,p}- \unif{n} \|_{\TV} = \Theta\left(\frac{1}{\sqrt{n}}\right).   
    \end{equation}
\end{customthm}
\begin{remark}
The above result for generic $p$ is in sharp contrast with the behavior of the model at $p=1/2$ for which the rate of convergence is $1/n$ by Theorem~\ref{thm:conv12}. Before we start the proof of the theorem, let us explain the flow of our argument. In Theorem 1 of \cite{bellin2021}, the author proved that for the classical model (corresponding to $p=1$), the rate of convergence of $(\| Q_{n,1}- \unif{n} \|_{\TV})_{n\ge 1}$ is $O(1/\sqrt{n})$. This result together with the second inequality of Proposition \ref{prop:dTVp} shows that in the probabilistic model under consideration, the rate of convergence is actually $O(1/\sqrt{n})$ for any $p\neq 1/2$. We then go one step further and establish that $1/\sqrt{n}$ is indeed the right order for the rate of convergence. Besides the first inequality of Proposition \ref{prop:dTVp}, a central ingredient in the proof is a clever choice of the test function $f$ as in Equation \eqref{test}.
\end{remark}
\begin{proof} For the upper bound, notice that by Proposition \ref{prop:dTVp} the rate of convergence in the general case is bounded from above by the corresponding rate in the classical case $p=1$. Then, taking $k=1$ in Theorem 1 of \cite{bellin2021}, we obtain that 
    \begin{equation}
    \| Q_{n,p}- \unif{n} \|_{\TV} = O\left( \frac{1}{\sqrt{n}}\right).
    \end{equation}
For the lower bound, define the function $f$ over $[n]$ such that $f(j) = j/n$ for all $j\in [n]$. Notice that by the definition of $f$ we have that
\begin{equation}\label{eq:uni}
    \sum_{j=1}^n f(j)\unif{n}(j) = \frac{n+1}{2n}.
\end{equation}
Whereas, by Theorem \ref{mean} we have
\begin{equation}\label{eq:qnp}
    \sum_{j=1}^nf(j)Q_{n,p}(j) = \frac{\mathbb{E}(a_n | \alpha \in \PF{n})}{n}=\frac{n+1}{2n}-(2p-1)\left[\frac{\sqrt{2\pi}}{4\sqrt{n}}-\frac{7}{6n}\right]+\frac{o(1)}{n}.
\end{equation}
Thus, for $p=1$, substitution \eqref{eq:uni} and \eqref{eq:qnp} into the identity given by Proposition~\ref{prop:altdTV} we have that 
\begin{equation}
    \begin{split}
    \| Q_{n,1}- \unif{n} \|_{\TV} & \ge \frac{1}{2}\left[\sum_{j=1}^n f(j)\unif{n}(j) -  \sum_{j=1}^nf(j)Q_{n,p}(j)\right] \\
    & = \frac{1}{2} \left[\frac{\sqrt{2\pi}}{4\sqrt{n}}-\frac{7}{6n}+\frac{o(1)}{n}\right].
    \end{split}
\end{equation}
Finally, by Proposition \ref{prop:dTVp} we know that 
\begin{equation}
\| Q_{n,p}- \unif{n} \|_{\TV} \ge |2p-1|\| Q_{n,1}- \unif{n} \|_{\TV} \sim |2p-1| \frac{\sqrt{2\pi}}{8} \frac{1}{\sqrt{n}},
\end{equation}
which proves the theorem.
\end{proof}

}

\subsection{A statistical interpretation}\label{subsec:stats}

Let us contextualize the Coin Problem with a more formal statistical framework. We again consider the problem of parking $n$ cars on a circle with $(n+1)$ spots. Suppose we have the random vector $\alpha = (a_1, a_2, \dots a_n)$, where each $a_i$ is drawn with replacement from $[n+1]$. That is, $\alpha$ denotes a random sample of parking preferences. Then, consider the following random variables:
\begin{itemize}
    \item Let $Y_i \sim g(\alpha | a_1, \dots, a_{i-1} ; p)$ be a random variable denoting the spot at which car $i$ parks. Notice that this distribution is parameterized by the forward probability $p \in [0, 1]$, and furthermore, is conditional on previous preferences.
    \item Let $T$ denote the vacant spot on the circle, after all $n$ cars have parked, similar to Definition~\ref{def:bucket}. Since $[n+1]$ denotes the spots on the circle, we have $T  = \sum_{i=1}^{n+1} i - \sum_{j = 1}^{n} Y_j$.
    \item Finally, let $Z = \mathbb{I}(\alpha \in S_{n+1}) = \mathbb{I}(T_i \ne n+1)$, where $\mathbb{I}$ is the indicator function. That is, $Z$ checks whether $\alpha$ is a parking function on the one-way street.
\end{itemize}

Reframing the results of Theorem~\ref{theorem:pollakcoin}, we have that $Z \sim \mathrm{Bernoulli}(\theta)$, for $\theta = (n+1)^{n-1}/n^n$. Similar to Theorem~\ref{theorem:pollakcoin}, the sampling distribution of $Z$ is independent of $p$. Then, if we applied the probabilistic parking protocol to each preference vector in $[n+1]^n$, treating each instance as an independent Bernoulli trial, we expect $(n+1)^{n-1}$ of those preference vectors to be parking functions. Finally, we emphasize that we precisely derived this Bernoulli distribution, without knowledge of the underlying conditional distribution $g(\boldsymbol{\cdot})$, which determines the cars' individual parking outcomes.

\subsection{The Coin Problem for $m \le n$ cars}\label{subsec:coin-m-n}

Having formalized the symmetric properties of parking on a circle in Section \ref{subsection:probabilityandparkingfunctions}, we now prove that Theorem~\ref{theorem:pollakcoin} generalizes to the problem of parking $m \in [n]$ cars on a one-way street with $n$ spots. Here, we slightly adjust the definition of a parking function to mean any preference vector in $[n]^m$ which assigns all $m$ cars to a unique spot on the one-way street with $n$ spots.

\begin{customthm}{2}
Consider the preference vector $\alpha \in [n]^m$, chosen uniformly at random. Then
\begin{equation}
\PP(\alpha \in \PF{n} | \alpha \in [n]^m) = \frac{(n+1-m)(n+1)^{m-1}}{n^m}.
\end{equation}
\end{customthm}

\begin{proof}
With $m \le n$ cars, it is possible that more than 1 spot is vacant after all $m$ cars park on a circle with spots $[n+1]$.  However, the general principles of symmetry, formalized by Lemma~\ref{lemma:pmf-modular} and Equation~\eqref{eqn:equal-column-row}, still apply. i.e., each preference vector in $[n+1]^m$ is modularly congruent to $n$ other preference vectors in $[n+1]^m$, namely those which are obtained by shifting all $m$ preferences by some $k \in [n]$. Then, the expected number of preference vectors in $[n+1]^m$ which leave spot $i$ vacant is equivalent for all $i \in [n+1]$, despite the fact that there are many more possible combinations of vacant spots.

Next, by symmetry, each of the $n+1$ spots has equal probability of being vacant. Since there are always $n+1-m$ vacant spots after all $m$ cars have parked, we know that the probability that an arbitrary spot in $[n+1]$ is vacant is $(n+1-m)/(n+1)$. Therefore, the expected number of preference vectors in $[n+1]^m$ which are parking functions is
\begin{equation}\label{eq:exp m<=n}
\mathbb{E}[\#\{\alpha \in [n+1]^m \cap \alpha \in \PF{n}\}] = \frac{n+1-m}{n+1} \cdot (n+1)^m = (n+1-m)(n+1)^{m-1}.
\end{equation}

But, generalizing the result of Lemma~\ref{lemma:does-not-contain}, we know that if a preference vector in $[n+1]^m$ is a parking function, then none of the $m$ cars may prefer spot $n+1$. Therefore, our effective sample space is the set of preference vectors in $[n]^m$. Since we choose uniformly at random, we divide the expectation in \eqref{eq:exp m<=n} by $n^m$, the size of this sample space, to obtain the probability in question. That is,
\begin{align*}
\PP(\alpha \in \PF{n} | \alpha \in [n]^m) = \frac{(n+1-m)(n+1)^{m-1}}{n^m}.
\end{align*}
\end{proof}

\section{Related combinatorial results} \label{section:relatedresults}

In this section, we present various combinatorial properties of parking, both on a one-way and on a circular street.

\subsection{Lucky and unlucky cars} \label{subsections:lucky}

Gessel and Seo~\cite{Gessel} defined a car being ``lucky" if it is able to park in its preferred spot. Similarly, we say that a car is ``unlucky" if it cannot park in its preferred spot. Notice that the first car is always lucky, since it enters an empty street. Therefore, we may use the convention that the leading preference is fixed to spot 1 when parking on a circular street.

\begin{customthm}{7}\label{thm:sumsofprods}
Suppose $n$ cars park on a circle with spots $[n+1]$. The expected number of preference vectors in $[n+1]^n$ containing $0 \le k < n$ unlucky cars is
\begin{eqnarray}\label{eqn:sumsofprods}
U_n(k) = n! \sum\limits_{C_k \vdash \{2,\dots,n\}} \prod\limits_{i \in C_k} \frac{i-1}{(n+1)-(i-1)},
\end{eqnarray}
\noindent where $C_k$ is a size-$k$ subset of the cars $\{2,\dots,n\}$, and each
car $i \in C_k$ is unable to park at its preferred spot.
\end{customthm}

\begin{proof}
First, consider the case for $k=0$, i.e., no cars are unlucky. Recall that the leading preference is fixed as spot 1. Here, there is only one unique approach to choose 0 unlucky cars among the remaining set of cars $\{2, \dots, n\}$. For car 2 to be lucky, it must prefer one of the $n$ vacant spots. Similarly, for car 3 to be lucky, it must prefer one of the $n-1$ vacant spots. Then, in general, there are $n(n-1) \cdots 2$ combinations of preferences which allow all cars to be lucky. Hence, $U_n(0) = n!$.

Next, let $C_k \vdash \{2,\dots,n\}$ denote an arbitrary combination of $k$ cars
designated as unlucky. When each car $i \in C_k$ enters the street, since it is
unlucky, it must prefer one of the $i-1$ occupied spots. That is, in the
descending product $n!$, for each $i \in C_k$, we replace the term $(n+1)-(i-1)$
with $i-1$. This is done algebraically by dividing $n!$ by $(n+1)-(i-1)$, then
multiplying by $i-1$. 
Summing over all possible combinations $C_k$ yields $U_n(k)$.
\end{proof}

For completion, we also relate this to the more traditional problem of parking $n$ cars on a one-way street with $n$ spots.
\begin{corollary}\label{coral:generating-one-way}
Suppose $n$ cars park on a one-way street with $n$ spots. The expected number of preference vectors in $[n]^n$ containing $0 \le k < n$ unlucky cars is
\begin{eqnarray}\label{eqn:generating-one-way}
\mathcal{U}_n(k) = n! \sum\limits_{C_k \vdash \{2,\dots,n\}} \prod\limits_{i \in
C_k} \frac{i-1}{n-(i-1)}.
\end{eqnarray}
\end{corollary}
\begin{proof}
The proof proceeds similarly to that for Theorem~\ref{thm:sumsofprods}, with the only difference being in the denominator, where we subtract 1, since each car may choose among $n$ spots only.
\end{proof}

\begin{remark}
This inner product is identical to the generating function derived by Khidr and El-Desouky~\cite{stirling}, which is based on the Stirling numbers of the first kind. This generating function produces the triangular array documented under OEIS \href{https://oeis.org/A071208}{A071208}.
\end{remark}

\subsection{Solving an open problem on polynomials} \label{subsection:oeis}

In this section, we present a solution which relates Theorem~\ref{thm:sumsofprods} to an open problem proposed by Novelli and Thibon~\cite{novelli2020duplicial}, in which the authors derived a generating polynomial and inquired about interpretations of the triangular array it produced. For reference, this sequence is documented under OEIS \href{https://oeis.org/A220884}{A220884}.

\begin{customthm}{8}\label{thm:oeis-triangle}
Let $Q_n(q)$ be the generating polynomial $\prod_{k=2}^n [(n+1-k)q + k]$. Then
\begin{equation}
Q_n(q) = \sum_{k=0}^{n-1} U_n(k) q^k.
\end{equation}
\end{customthm}

\begin{proof}
Recall that we fix the leading preference to be spot 1 on a circular street, since the first car is always lucky. Then, rewriting $\prod_{k=2}^n [(n+1-k)q + k] = \prod_{k=2}^{n} [(k-1)q + (n+1)-(k-1)]$, we see that the two coefficients in the $k$-th factor denote the number of preferences which make car $k$ unlucky and lucky, respectively. Therefore, when expanding the first $0 \le j \le n-1$ factors, i.e., $\prod_{k=2}^{j+1} [(k-1)q + (n+1)-(k-1)]$, and grouping together same-degree terms, each coefficient of $q^i$, for all $i \in \{0, \dots, j\}$, represents the expected number of \textit{partial} preference vectors in $[n+1]^{j+1}$ containing $i$ unlucky cars.
\end{proof}

\begin{remark}
In the left-aligned triangular array documented under OEIS \href{https://oeis.org/A220884}{A220884}, the entry in row $i$, column $j$ denotes the expected number of preference vectors in $[i+1]^i$ which have $j$ unlucky cars, with the first car parking at spot $1$ always. See Table \ref{tab:triangular}.
\end{remark}

\begin{table}[h]
\centering
\begin{floatrow}
\capbtabbox[1.3\linewidth]{%
 \scalebox{1.0}{\begin{tabular}{|c|c|c|c|c|c|c|} 
\hline
 \diagbox{$i$}{$j$}& $0$ & $1$ & $2$ & $3$ & $4$ & $5$ \\\hline
 $0$ & $1$   &      &       &       &    &    \\ 
\hline
$1$ & $1$   & $0$  &       &       &    &    \\ 
\hline
$2$ & $2$   & $1$   & $0$ &       &     &   \\ 
\hline
$3$ & $6$  & $8$  & $2$   &  $0$ &    &    \\ 
\hline
$4$ & $24$  & $58$  & $37$  & $6$   & $0$  & \\ 
\hline
$5$ & $120$ & $444$ & $504$ & $204$ & $24$ & $0$\\

\hline
\end{tabular}}}{%
   \caption{Triangular array under OEIS \href{https://oeis.org/A220884}{A220884}.}
    \label{tab:triangular}%
}
\end{floatrow}
\end{table}

\subsection{Weighted Pascal's triangle} \label{subsection:pascal}

The notion of partial preference vectors in the proof of Theorem~\ref{thm:oeis-triangle} alludes to a recurrence relation which bears very strong resemblance to that for Pascal's triangle.

\begin{definition}\label{def:recurrence relation}
Let $E_n(i,k)$, with $0 \le k < i \le n$, denote the expected number of
preference vectors in $[n+1]^i$ containing $k$ unlucky cars. Then
\begin{equation}\label{eqn:weightedpascal}
E_n(i,k)=((n+1)-(i-1))\cdot E_n(i-1,k)+(i-1)\cdot E_n(i-1,k-1).
\end{equation}
\end{definition}
We provide a sample calculation to show the mechanics of the recurrence relation of \eqref{eqn:weightedpascal}. For $n=4$, $i=3$ and $k=1$:
\begin{align*}
E_{4}(3,1)=(5-(3-1))\cdot E_4(2,1) +(3-1)\cdot E_4(2,0)
&=3\cdot1 + 2\cdot 4=3+8=11.
\end{align*}

We refer to the triangular array produced by \eqref{eqn:weightedpascal} as the ``Weighted Pascal's Triangle" due to the weighted coefficients. An example of this triangle, for a length $i=4$ sequence, is presented in Table~\ref{tab:pascalweighted}.

Notice how \eqref{eqn:weightedpascal} contrasts with the recurrence relation for Pascal's triangle: 
\begin{equation}\label{eqn:pascal}
T(n,d)=T(n-1,d)+T(n-1,d-1).
\end{equation}

Note that $n$ in this case represents the row of Pascal's triangle, and $d$ represents the column of the triangle. An example of such a triangle for $n=4$ is presented in Table~\ref{tab:pascal}. 

Notice that \eqref{eqn:weightedpascal} reduces to \eqref{eqn:pascal}, when both $n-i=1$ and $i=2$.

\begin{table}[h]
\centering
\hspace{-1in}
\begin{floatrow}
\capbtabbox[1.3\linewidth]{%
 \scalebox{1.00}{\begin{tabular}{|c|c|c|c|c|c|} 
\hline
\diagbox{$n$}{$d$} & $0$ & $1$ & $2$ & $3$ & $4$  \\\hline
$0$ & $1$   &       &       &  &      \\ 
\hline
$1$ & $1$   & $1$   &       &   &     \\ 
\hline
$2$ & $1$   & $2$   & $1$   &    &      \\ 
\hline
$3$ & $1$   & $3$   & $3$   & $1$   &       \\ 
\hline
$4$ & $1$   & $4$   & $6$   & $4$   & $1$   \\ 
\hline
\end{tabular}}
}{%
  \caption{Pascal's triangle for $0\leq d\leq n\leq 4$.}
    \label{tab:pascal}%
}
\hspace{-1in}
\capbtabbox[1.3\linewidth]{%
 \scalebox{1.0}{\begin{tabular}{|c|c|c|c|c|c|} 
\hline
 \diagbox{$i$}{$k$}& $0$ & $1$ & $2$ & $3$ & $4$  \\\hline
 $0$ & $0$   &       &       &       &        \\ 
\hline
$1$ & $1$   & $0$   &       &       &        \\ 
\hline
$2$ & $4$   & $1$   & $0$   &       &        \\ 
\hline
$3$ & $12$  & $11$  & $2$   & $0$   &        \\ 
\hline
$4$ & $24$  & $58$  & $37$  & $6$   & $0$    \\

\hline
\end{tabular}
}}{%
   \caption{Weighted Pascal's triangle for $n=4$.}
    \label{tab:pascalweighted}%
}
\end{floatrow}
\end{table}

It is useful to think about two limiting cases when discussing Table~\ref{tab:pascalweighted}. Observe the case of $i=0$ and $k=0$. The entry in the table/triangle is giving us the number of expected sequences with zero unlucky cars, when there are no cars parked: such a sequence does not exist because there are no parked cars on the street!
Another interesting limiting case occurs in the last row with $i = n$. Notice that each cell in this row can be computed with Theorem~\ref{thm:sumsofprods}. That is, $E_n(n,k) = U_n(k)$. For example, consider the case for $i=4$ and $k=0$. This entry tells us the expected number of sequences which have no unlucky cars, after four cars have parked, which is $U_4(0) = 4!$.

\section{Future research} \label{sec:futurework}
The combinatorial results derived in Section~\ref{section:relatedresults} connect to existing works which provide promising avenues for further analysis. For example, Diaconis and Hicks~\cite{diaconis} cite the generating function:
\begin{equation}\label{eq:last}
\sum_{\forall \alpha \in \PF{n}} q^{L(\alpha)} = q \prod_{i=1}^{n-1}[i+(n-i+1)q],
\end{equation}
where $L(\alpha)$ counts the number of lucky cars in the one-way classical parking function $\alpha \in \PF{n}$, chosen uniformly. In \cite{diaconis} it was shown that the cumulative density function of $L(\alpha)$ is asymptotically normal, which implies that the cumulative density function of $\sum U_n(k)$ as investigated in Theorem~\ref{thm:sumsofprods} is also asymptotically normal. The authors are further interested in similarly researching the distribution of the generating function for $U_n(k)$ for each valid $k$.

\section{Acknowledgements}

The authors thank Kevin Wang for his insights which made Section~\ref{subsec:coin-m-n} and Section~\ref{section:relatedresults} possible.
Pamela E. Harris and Mei Yin acknowledge helpful conversations in the Special Session on Algebraic, Geometric, and Topological Combinatorics at the 2022 AMS Central Fall Sectional Meeting in El Paso, organized by Duval Art, Caroline Klivans, and Jeremy L. Martin.
Rodrigo Ribeiro would like to thank Pedro Franklin from Federal University of Uberl\^{a}ndia for all the helpful conversations about statistics and simulations on Friday nights.

\bibliographystyle{plain}
\bibliography{Bibliography.bib}

\end{document}